\documentclass[leqno,10pt]{amsart}

\usepackage{amsmath,amssymb,mathrsfs} 

\theoremstyle{plain}
    \newtheorem{thm}{Theorem}[section]
    \newtheorem{ppn}[thm]{Proposition}
    \newtheorem{lem}[thm]{Lemma}
    \newtheorem{cor}[thm]{Corollary}
\theoremstyle{definition}
    \newtheorem{dfn}[thm]{Definition}
    
\theoremstyle{remark}
    \newtheorem{rmk}[thm]{Remark}    

\numberwithin{equation}{section}

\allowdisplaybreaks

\def\C{\mathbb{C}}\def\Q{\mathbb{Q}}\def\Z{\mathbb{Z}}\def\N{\mathbb{N}}

\def\a{\alpha}\def\b{\beta}\def\g{\gamma}\def\d{\delta}\def\pd{\partial}\def\e{\varepsilon}
\def\x{\xi}\def\y{\eta}
\def\o{\omega}

\def\sD{\mathscr{D}}

\def\s{\sigma}\def\t{\tau}

\def\vp{\varphi}

\def\y{\eta}

\def\D{\Delta}

\def\FF#1#2#3{{}_2F_1\left({#1\atop #2};#3\right)}

\def\r{\rho}
\def\bD{\mathbf{D}}

\begin{document}

\title{A new approach to hypergeometric transformation formulas}
\author{Noriyuki Otsubo}
\email{otsubo@math.s.chiba-u.ac.jp}
\address{Department of Mathematics and Informatics, Chiba University, Inage, Chiba, 263-8522 Japan}

\begin{abstract}
We give a new method to prove in a uniform and easy way various transformation formulas for Gauss hypergeometric functions. 
The key is Jacobi's canonical form of the hypergeometric differential equation. 
Analogy for $q$-hypergeometric functions is also studied. 
\end{abstract}

\date{\today}

\maketitle

\section{Introduction}

Recall that a Gauss hypergeometric function is defined by the power series
\begin{equation*}
\FF{a,b}{c}{x}=\sum_{n=0}^\infty \frac{(a)_n(b)_n}{(c)_n(1)_n}x^n, \quad (a)_n=\prod_{i=0}^{n-1}(a+i),
\end{equation*}
which converges on the open unit disk. Here the parameters $a$, $b$, $c$ are complex numbers and $-c\not\in\N$. 
We know many transformation formulas among such functions since Euler, Pfaff and Gauss, and some of them are quite new. Such formulas have various aspects and to find or prove them, different techniques have been used. 
For example, elliptic functions and computation using mathematical software played important roles. 

In this paper, we give a new method to prove such formulas in a uniform and easy way. 
Recall that the hypergeometric function satisfies a linear ordinary differential equation of order two, 
and hence is characterized by this equation together with the initial values.  
The key of our method is the following canonical form of the hypergeometric differential equation
\begin{equation*}
\left(\frac{d}{dx} x^c(1-x)^e \frac{d}{dx} -ab x^{c-1}(1-x)^{e-1}\right)y=0
\end{equation*}
where $e=1+a+b-c$ (Theorem \ref{p1}). 
As the author learned after writing the first manuscript of the present paper, this form was known by Jacobi \cite{jacobi} (see also \cite[\S 24]{poole}). 
Though it seems to have been scarcely used, at least unless $a+b=c=1$, it has several advantages to the standard ones (see \eqref{e0}, \eqref{e1}). 
It clarifies the symmetry under $x \leftrightarrow 1-x$ (then $c \leftrightarrow e$),  
and behaves nicely under the change of variables we are to consider. 
Above all, it enables us to compute the differential equation for $h(x)F(x)$ from that for $F(x)$ in a straight-forward way. 

Among the transformation formulas to be proved in this paper, of particular interest are the following ones, 
which have strong connection to number theory. 

\begin{thm}\label{thm1}On a neighborhood of $x=0$, 
\begin{align}
(1+x)^a \FF{\frac{a}{2},\frac{a-b+1}{2}}{\frac{b+1}{2}}{x^2}
&=\FF{\frac{a}{2},\frac{b}{2}}{b}{1-\left(\frac{1-x}{1+x}\right)^2}, \label{t2+}
\\
(1+2x)^a\FF{\frac{a}{3},\frac{a+1}{3}}{\frac{a+5}{6}}{x^3}
&=\FF{\frac{a}{3},\frac{a+1}{3}}{\frac{a+1}{2}}{1-\left(\frac{1-x}{1+2x}\right)^3}, \label{t3+}
\\
(1+3x)^{\frac{a}{2}} \FF{\frac{a}{4},\frac{a+2}{4}}{\frac{a+5}{6}}{x^2}
&=\FF{\frac{a}{4},\frac{a+2}{4}}{\frac{a+2}{3}}{1-\left(\frac{1-x}{1+3x}\right)^2}.\label{t4+}
\end{align}
\end{thm}

The quadratic formula \eqref{t2+} with two free parameters $a$ and $b$ is very classical, due to Gauss  \cite[p. 225, formula 101]{gauss}. 
The special case where $a=b=1$ reduces to the Landen transformation formula
$$(1+x)K(x)=K\left(\frac{2\sqrt{x}}{1+x}\right)$$
for the elliptic integral of the first kind
$$K(x):=\int_0^\frac{\pi}{2}\frac{d\theta}{\sqrt{1-x^2\cos^2\theta}}
=\frac{\pi}{2} \FF{\frac{1}{2},\frac{1}{2}}{1}{x^2}.$$
From this follows the formula 
$$\FF{\frac{1}{2},\frac{1}{2}}{1}{1-x^2}=\frac{1}{M(x)},$$
where $M(x)$ denotes the arithmetic-geometric mean of $1$ and $x$ ($0<x\le 1$). 

The cubic formula \eqref{t3+} for $a=1$ was first found by Ramanujan \cite[second notebook, p. 258]{ramanujan} in his study of elliptic functions to alternative bases. It was rediscovered and proved by Borwein-Borwein \cite[p. 694]{bb} in their study of a cubic analogue of the arithmetic-geometric mean. 
The general case of \eqref{t3+} is due to Berndt-Bhargava-Garvan \cite[Theorem 2.3]{b-b-g}.  

The quadratic formula \eqref{t4+} for $a=1$ was also found by Ramanujan \cite[p. 260]{ramanujan}. 
This is also related with a generalized arithmetic-geometric mean, and a proof was given implicitly by Borwein-Borwein \cite[Theorem 2.6]{bb} and explicitly by Berndt-Bhargava-Garvan \cite[Theorem 9.4]{b-b-g}. 
The general case of \eqref{t4+} was found recently by Matsumoto-Ohara \cite[Corollary 3]{matsumoto-ohara} (see Section \ref{ss-multi}). 

One may expect that our method is useful not only for proving formulas, but also for finding new ones. 
In fact, the author found \eqref{t4+} independently before knowing \cite{matsumoto-ohara}. 
For other transformation formulas which were found more recently, see for example Vid\=unas \cite{vidunas}. 

This paper is constructed as follows. 
In Section 2, we derive the canonical form of the hypergeometric equation and explain our general method for transformation formulas. 
In Section 3, we give a short proof of Theorem \ref{thm1}, and discuss transformation formulas 
for multivariable hypergeometric functions of Appell and Lauricella.  
In Sections 4--6, we give proofs of other quadratic, cubic and quartic formulas and discuss their relations with the previous formulas. 
In the last Section 7, we study the analogy for $q$-hypergeometric functions (basic hypergeometric functions) ${}_2\phi_1$.  
We give a canonical form of the $q$-hypergeometric difference equation (Theorem \ref{thm5}), and use it to give a new proof of Heine's transformation formula (Theorem \ref{thm6}). 

\section{Generalities}

\subsection{Hypergeometric differential equation}

Let us write differential operators as
$$\pd=\pd_x=\frac{d}{dx}, \quad D=D_x=x\pd.$$
One sees easily from $Dx^n=nx^n$ that 
$\FF{a,b}{c}{x}$ satisfies the differential equation
\begin{equation}\label{e0}
\left((D+a)(D+b)-x^{-1}D(D+c-1)\right)y=0.\end{equation}
Since $D^2=x^2 \pd^2+x\pd$, \eqref{e0} is equivalent to 
\begin{equation}\label{e1}
\left(x(1-x)\pd^2+(c-(1+a+b)x)\pd-ab\right)y=0. 
\end{equation}
Further, it can be written as
\begin{equation}\label{e2}
\left(\pd^2+\left(\frac{c}{x}-\frac{e}{1-x}\right)\pd-\frac{ab}{x(1-x)}\right)y=0, 
\end{equation}
where we define $e$ by 
$$1+a+b=c+e.$$
Then it is obvious that $\FF{a,b}{e}{1-x}$ is another solution ($-e\not\in\N$ assumed). 

\begin{rmk}
One cannot expect such a symmetry for the differential equation satisfied by 
a generalized hypergeometric function ${}_{p}F_{p-1}(x)$ in general for $p>2$. 
This can be seen from the asymmetry of the Riemann scheme.  
\end{rmk}

The key observation of this paper is the following. 
\begin{thm}\label{p1}
Put $\vp(x)=x^c(1-x)^e$ where $e=1+a+b-c$.  
Then  $\FF{a,b}{c}{x}$ {\rm(}resp. $\FF{a,b}{e}{1-x}${\rm )} is the unique solution of the differential equation 
\begin{equation}\label{e3}
\left(\pd \vp(x) \pd -ab\frac{\vp(x)}{x(1-x)}\right)y=0
\end{equation}
such that 
$$y(0)=1, \ y'(0)=\frac{ab}{c} \quad \left(\text{resp. } y(1)=1, \ y'(1)=-\frac{ab}{e}\right).$$
\end{thm}
\begin{proof}
For any function $f(x)$ regarded as a multiplication operator, the identity of operators
$$\pd f(x) = f(x) \pd + f'(x)$$ holds. 
Therefore, 
\begin{align*}
\pd \vp(x)
=&(x^c \pd +cx^{c-1})(1-x)^e
\\=& x^c((1-x)^e\pd-e(1-x)^{e-1})+cx^{c-1}(1-x)^e
\\=& \vp(x)\left(\pd+\frac{c}{x}-\frac{e}{1-x}\right).
\end{align*}
Hence follows the equivalence of \eqref{e2} and \eqref{e3}. 
The initial values are immediate from the definition and the uniqueness is evident. 
\end{proof}

\begin{rmk}In fact, the differential equation is regular singular at $x=0$, 
and $\FF{a,b}{c,}{x}$ (resp. $\FF{a,b}{e}{1-x}$) is the unique holomorphic solution with $y(0)=1$ (resp. $y(1)=1$). 
\end{rmk}

\subsection{Transformation}\label{ss-transformation}
We consider differential operators of the form 
$$\sD=\pd f(x) \pd - g(x).$$
If $F(x)$ is a solution of the differential equation $\sD y=0$, we say for brevity that 
$F(x)$ is a {\em solution of $\sD$}, and that $\sD$ is a {\em differential operator for $F(x)$}. 
This type of differential equation is stable under a change of variables. 

\begin{lem}\label{l2.2}
Let $F(x)$ be a solution of $\pd f(x) \pd - g(x)$ and 
$z(x)$ be a non-constant holomorphic function. Then 
$F(z(x))$ is a solution of
\begin{equation*}
\pd f(z(x))z'(x)^{-1}\pd -z'(x)g(z(x)).
\end{equation*}
\end{lem}
\begin{proof}Immediate from $\pd_z=z'(x)^{-1}\pd$. 
\end{proof}

Consider the differential operator as in \eqref{e3}
$$\sD=\pd\vp(x)\pd-ab\frac{\vp(x)}{x(1-x)}.$$
The following examples of $z(x)$ are important. 
First, for a positive integer $s$, let 
$$z(x)=x^s.$$  
Then $\sD$ becomes by Lemma \ref{l2.2}
\begin{equation}\label{e5}
\pd x^{sc-s+1}(1-x^s)^e \pd -s^2 ab x^{sc-1}(1-x^s)^{e-1}. 
\end{equation}

Secondly, for a positive integer $r$, let 
$$z(x)=\frac{1-x}{1+(r-1)x}.$$
Note that the map $x \mapsto z$ is an involution since 
$$(1+(r-1)x)(1+(r-1)z)=r. $$
We have
$$1-z=\frac{rx}{1+(r-1)x}, \quad \pd_z=-\frac{1}{r}(1+(r-1)x)^2\pd,$$
so $\sD$ becomes by Lemma \ref{l2.2}
\begin{equation}\label{e6}
\pd x^e(1-x)^c\r(x)^{-c-e+2} \pd -rab x^{e-1}(1-x)^{c-1}\r(x)^{-c-e}, 
\end{equation}
where we put
$$\r(x)=1+(r-1)x.$$ 

Finally, consider
$$z(x)=\left(\frac{1-x}{1+(r-1)x}\right)^s=\left(\frac{1-x}{\r(x)}\right)^s.$$
This is the composition of the two substitutions as above. 
Then,  
\begin{equation}\label{e7}
\pd_z=-\frac{1}{rs}(1-x)^{-s+1}\r(x)^{s+1}\pd.
\end{equation}
In general, the resulting differential operator is not so simple.  
For $(r,s)=(2,2)$, $(3,3)$ and $(4,2)$, however, we have respectively
\begin{align*}
1-z&=\frac{4x}{(1+x)^2}, \\
1-z& =\frac{9x(1+x+x^2)}{(1+2x)^3}=\frac{9x(1-x^3)}{(1-x)(1+2x)^3}, 
\\1-z&= \frac{8x(1+x)}{(1+3x)^2}=\frac{8x(1-x^2)}{(1-x)(1+3x)^2}.
\end{align*}
Hence the term $z^c(1-z)^e$ is a product of powers of $x$, $1-x$, $1-x^s$ and $\r(x)$.  
This explains why, in each formula of Theorem \ref{thm1}, the differential equation for the right-hand side is of a manageable form. 

\subsection{Comparison}\label{ss-comparison}
The formulas we prove are of the form
$$h(x)F_2(x)=F_1(x),$$
where $F_i(x)$ is a solution of a differential operator $\sD_i$ of order $2$ of the form 
$$\sD_i=\pd f_i(x)\pd - g_i(x).$$
Suppose that $F_1(x)$, $F_2(x)$  and $h(x)$ are holomorphic at $x=x_0$, $h(x_0)\ne 0$ 
and 
$$(hF_2)(x_0)=F_1(x_0), \quad (hF_2)'(x_0)= F_1'(x_0).$$ 
Then the equality $h(x)F_2(x)=F_1(x)$ holds on a neighborhood of $x_0$ if and only if $h(x)F_2(x)$ is also a solution of $\sD_1$. 

\begin{lem}
The identity of operators $h(x)\sD_1h(x)=\sD_2$ holds if and only if 
\begin{equation}\label{e8}
\begin{split}
f_2(x)&=f_1(x)h(x)^2, 
\\ g_2(x)& =g_1(x)h(x)^2-(f_1(x)h'(x))'h(x). 
\end{split}\end{equation}
\end{lem}

\begin{proof}
Using $\pd h-h\pd=h'$, we have an equality of differential operators 
\begin{align*}
& h\pd f_1 \pd h=(\pd h-h') f_1 (h\pd+h')=\pd f_1h^2\pd+\pd f_1 h' h - f_1h'h\pd - f_1h'^2\\
&=\pd f_1 h^2 \pd +(f_1 h'h)'-f_1h'^2= \pd f_1h^2 \pd + (f_1h')'h.  
\end{align*}
Hence 
$h\sD_1h= \pd f_1h^2 \pd + (f_1h')'h- g_1h^2$, 
and the lemma follows. 
\end{proof}

If the condition \eqref{e8} holds, then 
$$\sD_1h(x)F_2(x)=h(x)^{-1}\sD_2 F_2(x)=0,$$ 
hence $h(x)F_2(x)$ is a solution of $\sD_1$ near $x_0$. 
Conversely, if one seeks a transformation formula between $F_1(x)$ and $F_2(x)$, one is led to find a function $h(x)$ satisfying \eqref{e8}. 

\subsection{Linear Transformations}

As easy examples of our method, let us prove the following formulas, 
respectively due to Euler and Pfaff. 
Later, we prove a $q$-analogue of the former in a similar manner (see Theorem \ref{thm6}). 

\begin{thm}On a neighborhood of $x=0$, 
\begin{align}
(1-x)^{a+b-c}\FF{a,b}{c}{x}&=\FF{c-a,c-b}{c}{x},\label{tle}\\
(1-x)^{a}\FF{a,b}{c}{x}&=\FF{a,c-b}{c}{\frac{x}{x-1}}.\label{tlp}
\end{align}
\end{thm}

\begin{proof}
\eqref{tle}. Put $h=(1-x)^{a+b-c}$, $F_2=\FF{a,b}{c}{x}$ and $F_1=\FF{c-a,c-b}{c}{x}$. 
Using the notations of Section \ref{ss-comparison}, we have by Theorem \ref{p1} 
\begin{align*}
\sD_1&=\pd x^c(1-x)^{c-a-b+1} \pd - (c-a)(c-b)x^{c-1}(1-x)^{c-a-b},
\\ \sD_2&=\pd x^c(1-x)^{a+b-c+1} \pd - ab x^{c-1}(1-x)^{a+b-c}.
\end{align*}
The first equality of \eqref{e8} is obvious and for the second, 
\begin{align*}
&(f_1h')'h-g_1h^2
\\=&(a+b-c)cx^{c-1}(1-x)^{a+b-c}-(c-a)(c-b)x^{c-1}(1-x)^{a+b-c}
\\=&abx^{c-1}(1-x)^{a+b-c}=g_2.
\end{align*}
Since $F_1(0)=(hF_2)(0)=1$ and $F_1'(0)=(hF_2)'(0)=\frac{(c-a)(c-b)}{c}$, we have $hF_2=F_1$. 

\eqref{tlp}.  
Put $h=(1-x)^a$, $F_2=\FF{a,b}{c}{x}$ and $F_1=\FF{a,c-b}{c}{\frac{x}{x-1}}$. 
Then, by Theorem \ref{p1} and Lemma \ref{l2.2}, 
$$\sD_1=\pd x^c(1-x)^{-a+b-c+1}\pd + a(c-b)x^{c-1}(1-x)^{-a+b-c-1},$$
and $\sD_2$ is the same as \eqref{tle}.
The first equality of \eqref{e8} is obvious. For the second, 
since $f_1h'=-ax^c(1-x)^{b-c}$, we have using the logarithmic derivatives 
$$(f_1h')'h=f_1h'\left(\frac{c}{x}-\frac{b-c}{1-x}\right)h=
-ax^{c-1}(1-x)^{a+b-c-1}(c-bx).$$
Hence \eqref{e8} follows. The comparison of the initial values is easy, and \eqref{tlp} follows.  
\end{proof}

\section{Proof of Theorem \ref{thm1}}
Here we give direct proofs of the formulas \eqref{t2+}, \eqref{t3+} and \eqref{t4+}.  
Alternative proofs are given respectively in Section \ref{s-quadratic}, Section \ref{s-cubic} and Section \ref{s-quadratic}.  

\subsection{Proof of  \eqref{t2+}}
Let $F_1$ be the right-hand side and $F_2$ be the ${}_2F_1(x^2)$ in the left-hand side. 
Here, with the notations of Section \ref{ss-transformation}, $r=s=2$, $\r=1+x$, and 
$$z=\frac{(1-x)^2}{\r^2}, \quad 1-z=\frac{4x}{\r^2}, \quad \pd_z=-\frac{1}{4}\frac{\r^3}{1-x}\pd.$$
By Theorem \ref{p1} and Lemma \ref{l2.2}, we have with the notations of Section \ref{ss-comparison}
\begin{align*}
\sD_1&=\pd x^b(1-x)^{a-b+1}\r^{-a-b+1}\pd-abx^{b-1}(1-x)^{a-b+1}\r^{-a-b-1}
\\&= \pd x^b(1-x^2)^{a-b+1}\r^{-2a}\pd-abx^{b-1}(1-x^2)^{a-b+1}\r^{-2a-2}.
\end{align*}
On the other hand, we have by \eqref{e5}
$$\sD_2=\pd x^b(1-x^2)^{a-b+1}\pd-a(a-b+1)x^b(1-x^2)^{a-b}. $$
Letting $h=\r^a$, we have $f_1h^2=f_2$. Since 
$f_1h'=ax^b(1-x^2)^{a-b+1}\r^{-a-1}$, 
we have 
\begin{align*}(f_1h')'h
&=f_1h'\left(\frac{b}{x}-\frac{2(a-b+1)x}{1-x^2}-\frac{a+1}{1+x}\right)h
\\&=ax^{b-1}(1-x^2)^{a-b}\r^{-1}(b-(a+1)x-(a-b+1)x^2).
\end{align*}
Hence the condition \eqref{e8} is verified.  
Since $F_1(0)=(hF_2)(0)=1$ and $F_1'(0)=(hF_2)'(0)=a$, we obtain $hF_2=F_1$. 
\qed

\begin{rmk}
The original proof of Gauss also compares the differential equations. 
In Erd\'elyi et. al. \cite[p. 111, (5)]{erdelyi}, another proof is suggested but not explicitly. 
\end{rmk}

\subsection{Proof of \eqref{t3+}}
Let $F_1$ be the right-hand side and $F_2$ be the ${}_2F_1(x^3)$ in the left-hand side. 
Here, $r=s=3$, $\r=1+2x$, and 
$$z=\frac{(1-x)^3}{\r^3}, \quad 1-z=\frac{9x(1-x^3)}{(1-x)\r^3}, \quad \pd_z=-\frac{1}{9}\frac{\r^4}{(1-x)^2}\pd.$$
By Theorem \ref{p1} and Lemma \ref{l2.2}, we have
\begin{align*}
\sD_1&=\pd x^\frac{a+1}{2}(1-x^3)^\frac{a+1}{2}\r^{-2a}\pd-a(a+1)x^\frac{a-1}{2}(1-x)^2(1-x^3)^\frac{a-1}{2}\r^{-2a-2},
\\
\sD_2&=\pd x^\frac{a+1}{2}(1-x^3)^\frac{a+1}{2}\pd-a(a+1)x^\frac{a+3}{2}(1-x^3)^\frac{a-1}{2}. 
\end{align*}
Letting $h=\r^a$, we have
$$(f_1h')'h=a(a+1)x^\frac{a-1}{2}(1-x^3)^\frac{a-1}{2}\r^{-2}(1-2x-4x^3-4x^4).$$
One easily verifies the condition \eqref{e7} and the coincidence of the initial values. Hence we obtain $hF_2=F_1$. \qed

\begin{rmk}
See Section \ref{ss-multi} for another proof. 
When $a=1$, other proofs are given by Borwein-Borwein-Garvan \cite[Corollary 2.4]{bbg2}, Chan \cite[Sections 5 and 6]{chan}, Cooper \cite[Theorem 5.3]{cooper} and Maier \cite[Corollary 6.2]{maier}.  
\end{rmk}

\subsection{Proof of  \eqref{t4+}}
Let $F_1$ be the right-hand side and $F_2$ be the ${}_2F_1(x^2)$ in the left-hand side. 
Here, $r=4$, $s=2$, $\r=1+3x$ and 
$$z=\frac{(1-x)^2}{\r^2}, \quad 1-z=\frac{8x(1-x^2)}{(1-x)\r^2}, \quad \pd_z=-\frac{1}{8}\frac{\r^3}{1-x}\pd.$$
By Theorem \ref{p1} and Lemma \ref{l2.2}, we have
\begin{align*}
\sD_1&=\pd x^\frac{a+2}{3}(1-x^2)^\frac{a+2}{3}\r^{-a}\pd-\frac{a(a+2)}{2}x^\frac{a-1}{3}(1-x)(1-x^2)^\frac{a-1}{3}\r^{-a-2},
\\
\sD_2&=\pd x^\frac{a+2}{3}(1-x^2)^\frac{a+2}{3}\pd-\frac{a(a+2)}{4}x^\frac{a+2}{3}(1-x^2)^\frac{a-1}{3}. 
\end{align*}
Letting $h=\r^\frac{a}{2}$, we have
$$(f_1h')'h=\frac{a(a+2)}{4}x^\frac{a-1}{3}(1-x^2)^\frac{a-1}{3}\r^{-2}(2-3x-6x^2-9x^3).$$
Then, the rest is just as above. \qed

\begin{rmk}See Section \ref{ss-multi} for another proof. 
When $a=1$, other proofs are given by Cooper \cite[Theorem 5.3]{cooper} and Maier \cite[Corollary 6.2]{maier}.  
\end{rmk}

An investigation of our proofs suggests that, contrary to \eqref{t2+}, the formulas \eqref{t3+} and \eqref{t4+} have no generalization to a formula with two free parameters. 

\subsection{Multivariable hypergeometric functions}\label{ss-multi}

The formula \eqref{t3+} (resp. \eqref{t4+}) is obtained as a specialization of a transformation formula for a hypergeometric function of two (resp. three) variables. 
Recall Lauricella's hypergeometric function of $m$ variables \cite{lauricella} 
\begin{align*}& 
F_D^{(m)}(a,b_1,\dots,b_m;c;x_1,\dots, x_m)
\\& =
\sum_{n_1,\dots, n_m \ge 0} \frac{(a)_{n_1+\cdots+n_m}(b_1)_{n_1}\cdots (b_m)_{n_m}}{(c)_{n_1+\cdots+n_m}(1)_{n_1}\cdots(1)_{n_m}}x_1^{n_1}\cdots x_m^{n_m}.
\end{align*}
It converges on the open unit polydisk $\{(x_1,\dots, x_m) \in \C^m \mid \forall i, |x_i|<1\}$. 
When $m=1$, this is the Gauss function and when $m=2$, this is Appell's function $F_1$ \cite{appell}. 

For $m=2$ and $m=3$, we have the following  transformation formulas.  
\begin{equation}\label{emo1}
\begin{split}
&(1+x+y)^a F_D^{(2)}\left(\frac{a}{3},\frac{a+1}{6},\frac{a+1}{6};\frac{a+5}{6};x^3,y^3\right)
\\&=F_D^{(2)}\left(\frac{a}{3},\frac{a+1}{6},\frac{a+1}{6};\frac{a+1}{2};1-u^3,1-v^3\right), 
\end{split}
\end{equation}
where
$$u=\frac{1+\o x+ \o^2 y}{1+x+y}, \quad v=\frac{1+\o^2 x+ \o y}{1+x+y} \quad  (\o=e^\frac{2\pi i}{3}).$$ 

\begin{equation}\label{emo2}\begin{split}
&(1+x+y+z)^\frac{a}{2}F_D^{(3)}\left(\frac{a}{4},\frac{a+2}{12},\frac{a+2}{12},\frac{a+2}{12};\frac{a+5}{6};x^2,y^2,z^2\right)
\\&=F_D^{(3)}\left(\frac{a}{4},\frac{a+2}{12},\frac{a+2}{12},\frac{a+2}{12},\frac{a+2}{3};1-u^2,1-v^2,1-w^2\right), 
\end{split}\end{equation}
where 
$$u=\frac{1-x-y+z}{1+x+y+z}, \quad v=\frac{1-x+y-z}{1+x+y+z}, \quad w=\frac{1+x-y-z}{1+x+y+z}.$$
The formula \eqref{emo1} is due Koike-Shiga \cite[Proposition 2.5]{koike-shiga1} for $a=1$ and Matsumoto-Ohara \cite[Theorem 1]{matsumoto-ohara} in general. 
The formula \eqref{emo2} is due Kato-Mastumoto \cite[Proposition 1]{kato-matsumoto} for $a=1$ and Matsumoto-Ohara \cite[Theorem 3]{matsumoto-ohara} in general. 

Then we obtain \eqref{t3+} (resp. \eqref{t4+}) from \eqref{emo1} (resp. \eqref{emo2}) by letting $x=y$ (resp. $x=y=z$), using the multinomial formula 
$$\frac{(a_1+\cdots +a_m)_n}{(1)_n}=\sum_{i_1+\cdots+i_m=n}\frac{(a_1)_{i_1}\cdots (a_m)_{i_m}}{(1)_{i_1}\cdots (1)_{i_m}}.$$
Note $(a)_n/(1)_n=(-1)^n\binom{-a}{n}$. 

All the proofs in \cite{kato-matsumoto}, \cite{koike-shiga1} and \cite{matsumoto-ohara} use mathematical software to compare the two systems of partial differential equations. 
One might be able to give simpler proofs by extending the method of this paper. 
Put 
$$\pd_i=\frac{\pd}{\pd x_i}, \quad D_i=x_i\pd_i, \quad \bD=\sum_{i=1}^m D_i.$$ 
Then, $F_D^{(m)}$ is a solution of the system of linear partial differential equations 
(see \cite[Chapter 3, 9.1]{iwasaki})
\begin{align*}
& \left((\bD+a)(D_i+b_i)-x_i^{-1}D_i(\bD+c-1)\right)y=0 \quad (i=1,\dots, m), 
\\ & \left((D_i+b_i)x_j^{-1}D_j-(D_j+b_j)x_i^{-1}D_i\right)y=0 \quad (i, j=1,\dots, m). 
\end{align*}
In fact, this system is of rank $m+1$. 
Using Theorem \ref{p1}, we easily obtain  the following. 

\begin{thm}
Put 
$$\vp_i(x)=x^c(1-x)^{1+a+b_i-c} \quad (i=1,\dots, m), \quad \psi(x)=x^{c-1}(1-x)^{1+a-c}. $$
Then, $F_D^{(m)}(a,b_1,\dots,b_m;c;x_1,\dots, x_m)$ is the unique solution of 
\begin{align*}
& \left(\pd_i \vp_i(x_i) \pd_i -ab_i \frac{\vp_i(x_i)}{x_i(1-x_i)} + \psi(x_i)\pd_i (1-x_i)^{b_i}
\sum_{j\ne i} x_j\pd_j\right)y=0  \quad (i=1,\dots, m), 
\\& \left(\pd_ix_i^{b_i}x_j^{b_j-1}\pd_j- \pd_jx_j^{b_j}x_i^{b_i-1}\pd_i\right)y=0 \quad 
(i,j=1,\dots, m), 
\end{align*}
such that
$$y(0,\dots, 0)=1, \quad (\pd_ i y)(0,\dots, 0)=\frac{a b_i}{c} \quad (i=1,\dots, m).$$
\end{thm}

\section{Quadratic Formulas}\label{s-quadratic}

The formulas \eqref{t2+} and \eqref{t4+} are two different combinations of the following formulas, 
the former due to Kummer \cite[p. 78, Formula 44]{kummer}, and the latter due to Ramanujan  (cf. \cite[p. 50, Chap. 11, Entry 4]{notebook2}). 
These are in fact equivalent as is clear from the proof below. 

\begin{thm}\label{thm2}
On a neighborhood of $x=0$, 
\begin{align}
(1+x)^a\FF{\frac{a}{2},\frac{a+1}{2}}{b+\frac{1}{2}}{x^2}
&=\FF{a,b}{2b}{1-\frac{1-x}{1+x}}, \label{tk}\\
(1+x)^a\FF{a,b}{a-b+1}{x}
&=\FF{\frac{a}{2},\frac{a+1}{2}}{a-b+1}{1-\left(\frac{1-x}{1+x}\right)^2}. \label{tr}
\end{align}
\end{thm}

\begin{proof}\eqref{tk}. Let $F_1$ be the right-hand side and $F_2$ be the ${}_2F_1(x^2)$ in the left-hand side.  
By \eqref{e6} with $r=2$, the differential operator for $F_1$ is 
\begin{align*}
\sD_1&=\pd x^{2b}(1-x)^{a-b+1}\r^{-a-b+1}\pd-2abx^{2b-1}(1-x)^{a-b}\r^{-a-b-1}.
\\& =\pd x^{2b}(1-x^2)^{a-b+1}\r^{-2a}-2abx^{2b-1}(1-x^2)^{a-b}\r^{-2a-1}\pd, 
\end{align*}
where $\r=1+x$. 
On the other hand, by \eqref{e5} with $s=2$, the differential operator for $F_2$ is 
$$\sD_2=\pd x^{2b}(1-x^2)^{a-b+1}\pd-a(a+1)x^{2b}(1-x^2)^{a-b}.$$
Letting $h=\r^a$ so that $f_1h^2=f_2$ (with the notations of Section \ref{ss-comparison}), 
we have 
$$(f_1h')'h=ax^{2b-1}(1-x^2)^{a-b}\r^{-1}(2b-(a+1)x-(a+1)x^2),$$
and the condition \eqref{e8} holds.  
Comparing the initial values, we obtain \eqref{tk}. 

\eqref{tr}. If we replace $x$ with $\frac{1-x}{1+x}$, the formula becomes
$$\left(\frac{1+x}{2}\right)^a\FF{\frac{a}{2},\frac{a+1}{2}}{a-b+1}{1-x^2}
=\FF{a,b}{a-b+1}{\frac{1-x}{1+x}}.$$
By Theorem \ref{p1}, the both sides  
satisfy the same differential equation as those of \eqref{tk}.  
Comparing the initial values at $x=1$, we obtain the equality above, hence \eqref{tr}.  
\end{proof}

Now, let us deduce \eqref{t2+} and \eqref{t4+} from Theorem \ref{thm2}. 
Rewrite \eqref{tk} and \eqref{tr} as
\begin{align}
(1+u)^a\FF{\frac{a}{2},\frac{a+1}{2}}{b+\frac{1}{2}}{u^2}&=\FF{a,b}{2b}{1-x}, \label{tk2}\\
\FF{\frac{a}{2},\frac{a+1}{2}}{a-b'+1}{1-v^2}&=(1+y)^a\FF{a,b'}{a-b'+1}{y}, \label{tr2}
\end{align}
where 
$$(1+x)(1+u)= (1+y)(1+v)=2.$$

First, if we let $x=\x^2$, $y=\y^2$ and $(1+\x)(1+\y)=2$, then 
$$u^2+v^2=\left(\frac{1-\x^2}{1+\x^2}\right)^2
+\left(\frac{1-(\frac{1-\x}{1+\x})^2}{1+(\frac{1-\x}{1+\x})^2}\right)^2=1.$$
Letting $b'=a-b+\frac{1}{2}$ and equating the left-hand sides of \eqref{tk2} and \eqref{tr2}, 
we obtain
$$(1+u)^{a}(1+y)^a\FF{a,a-b+\frac{1}{2}}{b+\frac{1}{2}}{y}=\FF{a,b}{2b}{1-x}.$$
Since $(1+u)(1+y)=(1+\y)^2$,  
it becomes \eqref{t2+} (in variable $\y$) after a suitable change of parameters. 

Secondly, if we let $x+y=1$, then 
$$v=\frac{1-y}{1+y}=\frac{x}{2-x}=\frac{\frac{1-u}{1+u}}{2-\frac{1-u}{1+u}}=\frac{1-u}{1+3u}.$$
Letting $b=b'=\frac{a+1}{3}$ and equating the right-hand sides of \eqref{tk2} and \eqref{tr2}, we obtain
$$(1+y)^a(1+u)^a \FF{\frac{a}{2},\frac{a+1}{2}}{\frac{2a+5}{6}}{u^2}
=\FF{\frac{a}{2},\frac{a+1}{2}}{\frac{2a+2}{3}}{1-v^2}. $$
Since $(1+y)(1+u)=1+3u$, it becomes \eqref{t4+} (in variable $u$) after a suitable change of parameters. 

Let us give short proofs of two other important formulas. 
First, the following is due to Gauss \cite[p. 226, Formula 102]{gauss}. 
\begin{thm}On a neighborhood of $x=0$, 
\begin{equation}\label{t8}
\FF{a,b}{\frac{a+b+1}{2}}{x}
=\FF{\frac{a}{2},\frac{b}{2}}{\frac{a+b+1}{2}}{1-(1-2x)^2}.
\end{equation}
\end{thm}

\begin{proof}Letting $z=(1-2x)^2$, we have 
$1-z=4x(1-x)$, $\pd_z=-\frac{1}{4(1-2x)}\pd$. 
By Theorem \ref{p1} and Lemma \ref{l2.2}, the differential operator for the right-hand side is 
$$\pd x^\frac{a+b+1}{2}(1-x)^\frac{a+b+1}{2}\pd -abx^\frac{a+b-1}{2}(1-x)^\frac{a+b-1}{2},$$
which coincides with the differential operator for the left-hand side. Comparing the initial values, we obtain \eqref{t8}.
\end{proof}

The following is due to Kummer \cite[Formula 53]{kummer}.
\begin{thm}
On a neighborhood of $x=0$, 
\begin{equation}\label{t9}
(1+x)^a\FF{a,\frac{a-b+1}{2}}{\frac{a+b+1}{2}}{-x}
=\FF{\frac{a}{2},\frac{b}{2}}{\frac{a+b+1}{2}}{1-\left(\frac{1-x}{1+x}\right)^2}.
\end{equation}
\end{thm}

\begin{proof}
Let $z=(\frac{1-x}{1+x})^2$ as in the proof of \eqref{t2+}. 
By Theorem \ref{p1} and Lemma \ref{l2.2}, the differential operator for the right-hand side is 
$$\sD_1=\pd x^\frac{a+b+1}{2}(1+x)^{-a-b+1}\pd-abx^\frac{a+b-1}{2}(1+x)^{-a-b-1}.$$
Similarly, the differential operator for the ${}_2F_1(-x)$ in the left-hand side is 
$$\sD_2=\pd x^\frac{a+b+1}{2}(1+x)^{a-b+1} \pd + \frac{a(a-b+1)}{2}x^\frac{a+b-1}{2}(1+x)^{a-b}.$$
Letting $h=(1+x)^a$, one verifies \eqref{e8}. 
Comparing the initial values, we obtain \eqref{t9}.
\end{proof}

\begin{rmk}
In fact, \eqref{t8} and \eqref{t9} are equivalent to each other; apply Pfaff's formula \eqref{tlp} to the left-hand side of \eqref{t8} 
and then replace $x$ with $\frac{x}{1+x}$, to obtain \eqref{t9}. 
\end{rmk}

\section{Cubic Formulas}\label{s-cubic}

A cubic analogue of Theorem \ref{thm2} is the following formulas due to Goursat \cite[p. 140, (127) and (126)]{goursat}. 
Here we give short proofs and see that 
two different combinations of \eqref{tg1} and \eqref{tg2} give the formulas \eqref{t3+} and \eqref{t3.2}. 

\begin{thm}\label{thm3}
On a neighborhood of $x=0$, 
\begin{equation}\label{tg1}
(1+8x)^a\FF{\frac{4a}{3},\frac{4a+1}{3}}{\frac{4a+5}{6}}{x}
=\FF{\frac{a}{3},\frac{a+1}{3}}{\frac{4a+5}{6}}{64x\left(\frac{1-x}{1+8x}\right)^3}.
\end{equation}
On a neighborhood of $x=1$, 
\begin{equation}\label{tg2}
\left(\frac{1+8x}{9}\right)^a \FF{\frac{4a}{3},\frac{4a+1}{3}}{\frac{4a+1}{2}}{1-x}
=\FF{\frac{a}{3},\frac{a+1}{3}}{\frac{4a+5}{6}}{64x\left(\frac{1-x}{1+8x}\right)^3}. 
\end{equation}
\end{thm} 
 
\begin{proof}
If we let $z=64x\bigl(\frac{1-x}{1+8x}\bigr)^3$, then
$$1-z=\frac{\t^2}{\r^3}, \quad \pd_z=-\frac{\r^4}{64(1-x)^2\t}\pd,$$
where we put 
$$\r=1+8x, \quad \t=1-20x-8x^2.$$ 
By Theorem \ref{p1} and Lemma \ref{l2.2}, the differential operator for the right-hand sides of \eqref{tg1} and \eqref{tg2} is 
$$\sD_1=\pd x^\frac{4a+5}{6}(1-x)^\frac{4a+1}{2}\r^{-2a}\pd-\frac{64a(a+1)}{9}x^\frac{4a-1}{6}(1-x)^\frac{4a+3}{2}\r^{-2a-2}.$$
Letting $h=\r^a$, we have
$$(f_1h')'h=\frac{4a}{3}x^\frac{4a-1}{6}(1-x)^\frac{4a-1}{2}\r^{-2}(4a+5-16(2a+1)x-16(1+5a)x^2),$$
and then
$$g_1h^2-(f_1h')'h=\frac{4a(4a+1)}{9}x^\frac{4a-1}{6}(1-x)^\frac{4a-1}{2}. $$
Therefore, by Theorem \ref{p1} and \eqref{e8}, $h\sD_1h$ is nothing but the differential operator for 
${}_2F_1(x)$ (resp. for ${}_2F_1(1-x)$) in the left-hand side of \eqref{tg1} (resp. \eqref{tg2}). 
Comparing the initial values at $x=0$ (resp. $x=1$), we obtain \eqref{tg1} (resp. \eqref{tg2}). 
\end{proof}

As was found by Chan \cite[Section 6]{chan}, \eqref{t3+} is deduced from Theorem \ref{thm3} as follows. 
If we let $x=\x^3$, $y=\y^3$ and $(1+2\x)(1+2\y)=3$, then 
$$x\left(\frac{1-x}{1+8x}\right)^3=y\left(\frac{1-y}{1+8y}\right)^3, \quad 
\frac{9(1+8x)}{1+8y}=(1+2\x)^4.$$
Equating the right-hand sides of \eqref{tg1} and \eqref{tg2} and letting $a \to a/4$, we obtain \eqref{t3+} (in variable $\xi$).  

On the other hand, by equating the left-hand sides of \eqref{tg1} and \eqref{tg2}, 
which is only possible for $a=\frac{1}{4}$, we obtain
\begin{align*}
&\left(\frac{9(1+8x)}{9-8x}\right)^{\frac{1}{4}}\FF{\frac{1}{12},\frac{5}{12}}{1}{64(1-x)\left(\frac{x}{9-8x}\right)^3}
\\
&=\FF{\frac{1}{12},\frac{5}{12}}{1}{64x\left(\frac{1-x}{1+8x}\right)^3}
\end{align*}
on a neighborhood of $x=0$. 
In particular, the both sides satisfy the same differential equation, which remains true near $x=1$. 
Replacing $x$ with $\frac{9x}{1+8x}$ and comparing the initial values at $x=1$, we obtain the following.

\begin{cor}\label{c1} 
On a neighborhood of $x=0$ (resp. $x=1$), 
\begin{equation}\label{t3.2}
(1+80x)^\frac{1}{4}\FF{\frac{1}{12},\frac{5}{12}}{1}{64 x^3 \frac{1-x}{1+8x}}
=C \FF{\frac{1}{12},\frac{5}{12}}{1}{64\frac{9x}{1+8x} \left(\frac{1-x}{1+80x}\right)^3}, 
\end{equation}
where $C=1$ (resp. $C=3$). 
\end{cor}
The author does not know if it is equivalent to a known formula. 

\section{Quartic formulas}

Here we treat two quartic formulas. 
First, the following is an iteration of \eqref{t2+}.  

\begin{cor}\label{cor2}
On a neighborhood of $x=0$, 
\begin{equation}\label{t41}
(1+x)^{2a} \FF{\frac{a}{2},\frac{2a+1}{6}}{\frac{a+5}{6}}{x^4}
=\FF{\frac{a}{2},\frac{2a+1}{6}}{\frac{2a+1}{3}}{1-\left(\frac{1-x}{1+x}\right)^4}.
\end{equation}
\end{cor}

\begin{proof}
As in the proof of \eqref{t2+} given in Section \ref{s-quadratic}, let $x=\x^2$, $y=\y^2$ and 
$$(1+\x)(1+\y)=(1+x)(1+u)=(1+y)(1+v)=2,$$
so that $u^2+v^2=1$.  
Solving $(\frac{a}{2},\frac{b}{2},b)=(\frac{a'}{2},\frac{a'-b'+1}{2}, \frac{b'+1}{2})$, we have $b=\frac{a+2}{3}$, 
$(\frac{a'}{2},\frac{b'}{2},b')=(\frac{a}{2},\frac{2a+1}{6},\frac{2a+1}{3})$. 
Then by \eqref{t2+}, we have
\begin{align*}
(1+\x^2)^a\FF{\frac{a}{2},\frac{2a+1}{6}}{\frac{a+5}{6}}{\x^4}&=\FF{\frac{a}{2},\frac{a+2}{6}}{\frac{a+2}{3}}{1-u^2}
\\=\FF{\frac{a}{2},\frac{a+2}{6}}{\frac{a+2}{3}}{v^2}&=\left(\frac{1+\y^2}{2}\right)^a \FF{\frac{a}{2},\frac{2a+1}{6}}{\frac{2a+1}{3}}{1-\y^4}. 
\end{align*}
Since $\frac{1+\y^2}{2}=\frac{1+\x^2}{(1+\x)^2}$, we obtain \eqref{t41} (in variable $\x$). 
\end{proof}

The following formula of Matsumoto-Ohara \cite[Corollary 2]{matsumoto-ohara} is a specialization of a transformation formula (loc. cit. Theorem 2) for Appell's function $F_1$, whose proof is similar to \eqref{emo1} and \eqref{emo2}. 
We give a direct proof.

\begin{thm}
On a neighborhood of $x=0$, 
\begin{equation}\label{t10}
(1+x)^a\FF{\frac{a}{2},\frac{a+1}{4}}{\frac{a+3}{4}}{-x^2}
=\FF{\frac{a}{4},\frac{a+1}{4}}{\frac{a+1}{2}}{1-\left(\frac{1-x}{1+x}\right)^4}.
\end{equation}
\end{thm}

\begin{proof}By Theorem \ref{p1} and \eqref{e7}, the differential operator $\sD_1$ 
for the right-hand side is given by
\begin{align*}
f_1&=\left(x(1+x^2)\right)^\frac{a+1}{2}(1+x)^{-2a}, \\ 
g_1&=\frac{a(a+1)}{2}\left(x(1+x^2)\right)^\frac{a-1}{2}\left(\frac{1-x}{1+x}\right)^2(1+x)^{-2a}.
\end{align*}
On the other hand, by Theorem \ref{p1} and Lemma \ref{l2.2} applied to $z(x)=-x^2$, 
the differential operator $\sD_2$ 
for the ${}_2F_1(-x^2)$ in the left-hand side is given by
$$f_2=(x(1+x^2))^\frac{a+1}{2}, \quad 
g_2=-\frac{a(a+1)}{2}x^\frac{a+1}{2} (1+x^2)^\frac{a-1}{2}.$$
Letting $h=(1+x)^a$, one easily verifies \eqref{e8}. Comparing the initial values, we obtain \eqref{t10}.  
\end{proof}

\begin{rmk}The author learned from Hiroyuki Ochiai that \eqref{t10} is also obtained as a combination of \eqref{t2+} and \eqref{t9} as follows. 
Use the same notations as in the proof of Corollary \ref{cor2}. 
By \eqref{t2+} with $(a,b) \to (\frac{a}{2}, \frac{a+1}{2})$, 
$$\FF{\frac{a}{4},\frac{1}{4}}{\frac{a+3}{4}}{u^2}=\left(\frac{1+x}{2}\right)^\frac{a}{2}\FF{\frac{a}{4},\frac{a+1}{4}}{\frac{a+1}{2}}{1-x^2}.$$
By \eqref{t9} with $(a,b)\to (\frac{a}{2},\frac{1}{2})$, 
$$\FF{\frac{a}{4},\frac{1}{4}}{\frac{a+3}{4}}{1-v^2}
=(1+y)^\frac{a}{2}\FF{\frac{a}{2},\frac{a+1}{4}}{\frac{a+3}{4}}{-y}.$$
Then \eqref{t10} (in variable $\y$) follows similarly as \eqref{t41}. 
\end{rmk}

\section{$q$-analogues}

We give a new canonical form of the difference equation for a $q$-hypergeometric series ${}_2\phi_1$ which generalizes \eqref{e3}, and apply it to give a proof of Heine's transformation formula.  

\subsection{Preliminaries}
For the moment, let $\a$, $\b$, $\g$ and $q$ be indeterminates. 
Recall the {\em $q$-Pochhammer symbol}
$$(\a;q)_n=\prod_{i=0}^{n-1}(1-\a q^i) \ \in \Z[\a,q]\cap \Z[\a][[q]]^*.$$
Here, for a ring $R$, $R^*$ denotes its unit group. 
The $q$-hypergeometric series ${}_2\phi_1$ is defined by 
$${}_2\phi_1\left({\a,\b\atop \g};x\right)=\sum_{n=0}^\infty \frac{(\a;q)_n(\b;q)_n}{(\g;q)_n(q;q)_n} x^n. $$
This is a power series in $x$ with coefficients in $\Q(\a,\b,\g,q) \cap \Z[\a,\b,\g][[q]]^*$. 

Recall the {\em $q$-number}  
$$[n]=\frac{1-q^n}{1-q} \quad (n \in \Z).$$
Note that $[n]|_{q=1}=n$. 
We write $\a=q^a$ symbolically and define the number 
$$[a]=\frac{1-q^a}{1-q}=\frac{1-\a}{1-q} \quad \in \Q(\a,q)\cap \Z[\a][[q]]^*.$$

Define a difference operator ($q$-derivation) $\D$ by 
$$\D f(x)= \frac{f(x)-f(qx)}{x-qx}.$$
Following Jackson \cite{jackson}, define the shift operator $q^\d$ by 
$$q^\d f(x)=f(qx),$$
and the difference operator by
$$[\d+a]=\frac{1-q^{\d+a}}{1-q}=\frac{1-\a q^\d}{1-q}.$$
In particular, we have by definition 
$$[\d]=x\D.$$
Then we have 
$$[\d+a] x^n=\frac{1-\a q^n}{1-q}x^n =[a+n]x^n.$$
Hence $[\d+a]$ is the $q$-analogue of the differential operator $D+a$, where $D=x\frac{d}{dx}$. 

Any function (or a power series) $f(x)$ defines a multiplication operator.    
To avoid possible confusion in writing operators, we write 
$$f'(x)=\D f(x) \text{ (the function $f(x)$ acted by $\D$).}$$
As an operator, $\D f(x)$ means the composition of $f(x)$ and $\D$. 

\begin{lem}\label{l7.1}
For any $f(x)$, we have an identity of operators
$$\D f(x)= f(qx)\D+f'(x).$$
\end{lem}

\begin{proof}Since
\begin{align*}
&(f(x)g(x))'=\frac{f(x)g(x)-f(qx)g(qx)}{x-qx}
\\&=f(qx)\frac{g(x)-g(qx)}{x-qx}+\frac{f(x)-f(qx)}{x-qx}g(x)=f(qx)g'(x)+f'(x)g(x)
\end{align*}
for any $g(x)$, the lemma follows. 
\end{proof}

\subsection{$q$-hypergeometric difference equation}

Recall the differential equation \eqref{e0} for $\FF{a,b}{c}{x}$.  
Similarly, since
$$[\d+a]x^n= \frac{(\a;q)_{n+1}}{(1-q)(\a;q)_n}x^n, $$
the series ${}_2\phi_1\left({\a,\b\atop \g};x\right)$ satisfies the difference equation
\begin{equation}\label{e0q}
\left([\d+a][\d+b]-x^{-1}[\d][\d+c-1]\right)y=0, 
\end{equation}
where we used symbols $\a=q^a$, $\b=q^b$ and $\g=q^c$. 
From this, we derive difference equations analogous to \eqref{e1} and \eqref{e2}. 

\begin{ppn}
Put $\e=q\a\b\g^{-1}$. Then 
${}_2\phi_1\left({\a,\b\atop \g};x\right)$ is a solution of the difference equation 
\begin{equation}\label{e1q}
\Bigl(\g x(1-\e x) \D^2 +\bigl([c]-(\a\b+\b[a]+\a[b])x\bigr)\D-[a][b]\Bigr)y=0, 
\end{equation}
or equivalently
\begin{equation}\label{e2q}
\left(\D^2+\left(\frac{[c]}{\g x}
-\frac{\a\b+\b[a]+\a[b]-\e[c]}{\g(1-\e x)}\right)\D
-\frac{[a][b]}{\g x(1-\e x)}\right)y=0.
\end{equation}
\end{ppn}

\begin{proof}In \eqref{e0q}, substitute 
$$[\d+a]=\frac{1-\a q^\d}{1-q}=\a\frac{1-q^\d}{1-q}+\frac{1-\a}{1-q}=\a [\d]+[a].$$
By Lemma \ref{l7.1}, we have $\D x=qx \D+1$ (operators), hence
$$[\d]^2=x\D x\D=qx^2\D^2+x\D.$$
Then, using $[c-1]+q^{-1}\g=[c]$, we obtain \eqref{e1q}, hence \eqref{e2q}. 
\end{proof}

Now, let $q \in \C$ with $0<|q|<1$.  
Let $a$, $b$, $c \in \C$, $\a=q^a$, $\b=q^b$, $\g=q^c$ and assume  $\g \not\in q^{-\N}$. Then the power series ${}_2\phi_1\left({\a,\b \atop \g}; x\right) \in \C[[x]]$ defines an analytic function on $|x|<1$.  
Note that, in the limit as $q \to 1$, we have 
$\a, \b, \g \to 1$, $[a], [b], [c]  \to a, b, c$, hence ${}_2\phi_1\left({\a,\b \atop \g}; x\right) \to \FF{a,b}{c}{x}$. 
Since $[\d+a]\mapsto D+a$ and $\D \to \pd$, the equation \eqref{e0q} (resp. \eqref{e1q}, \eqref{e2q}) specializes to \eqref{e0} (resp. \eqref{e1}, \eqref{e2}). 

We give a $q$-analogue of \eqref{e3}. 
The function $x^a$ satisfies
\begin{equation}\label{e10}
\D x^a = [a]x^{a-1}.
\end{equation}
A $q$-analogue of the function $(1-x)^a$ is the following. 

\begin{dfn}For $\a=q^a$, define a function by 
$$\phi_\a(x)=\frac{(x;q)_\infty}{(\a x;q)_\infty},$$
where $(x;q)_\infty=\prod_{i=0}^\infty (1-xq^i)$. It converges on $|x|<1$. 
\end{dfn}

Recall that
${}_1F_0\left({a\atop \ };x\right)=(1-x)^{-a}$. 
Similarly, we have the $q$-binomial theorem 
$$ {}_1\phi_0\left({\a\atop \ };x\right)
:=\sum_{n=0}^\infty \frac{(\a;q)_n}{(q;q)_n} x^n
=\phi_\a(x)^{-1}$$
(see for example \cite[(1.3.2)]{gr}). 
Moreover, one sees easily the following. 

\begin{lem}\label{l7.2}We have
\begin{align*}
& \phi_q(x)=1-x, 
\\& \phi_{\a\b}(x)=\phi_\a(\b x)\phi_\b(x) =\phi_\a(x)\phi_\b(\a x), 
\\& \phi_\a'(x)=-[a]\frac{\phi_\a(x)}{1-x}=-[a]\phi_{q^{-1}\a}(qx).
\end{align*}
\end{lem}

Our canonical form of the difference equation for ${}_2\phi_1$ is the following. 
\begin{thm}\label{thm5}
Put $\vp(x)=x^c \phi_\e(x)$ where $\e=q\a\b\g^{-1}$. Then ${}_2\phi_1\left({\a,\b \atop \g};x\right)$ is the unique solution of the difference equation 
\begin{equation}\label{e3q}
\left(\D \vp(x) \D+(1-q)[a][b]\frac{\vp(x)}{1-x}\D-[a][b]\frac{\vp(x)}{x(1-x)}\right)y=0 
\end{equation}
such that
$$y(0)=1, \quad (\D y)(0)=\frac{[a][b]}{[c]}.$$
\end{thm}

\begin{proof}By Lemma \ref{l7.1}, \eqref{e10} and Lemma \ref{l7.2}, we have identities of operators 
\begin{align*}
\D \vp(x)& =(\g x^c\D+[c]x^{c-1})\phi_\e(x)
\\ &=\g x^c(\phi_\e(qx)\D+\phi_\e'(x))+[c]x^{-1}\vp(x)
\\& =\g\vp(x) \left(\frac{1-\e x}{1-x} \D -\frac{[e]}{1-x}+\frac{[c]}{\g x}\right)
 \\&=\g\vp(x) \frac{1-\e x}{1-x}\left(\D+\frac{[c]}{\g x}-\frac{[e]}{\g(1-\e x)} \right). 
\end{align*}
Then we obtain \eqref{e3q} from \eqref{e2q}, using 
\begin{align*}
[e]-(1-q)[a][b]&=\frac{\a\b(1-q)+\b(1-\a)+\a(1-\b)-\e(1-\g)}{1-q}
\\& =\a\b+\b[a]+\a[b]-\e[c]. 
\end{align*}
The initial condition follows by 
\begin{equation*}
\D {}_2\phi_1\left({\a,\b \atop \g};x\right)
=\frac{[a][b]}{[c]} {}_2\phi_1\left({q\a,q\b \atop q\g};x\right), 
\end{equation*}
and the uniqueness is evident. 
\end{proof}

\subsection{Transformation}

The following transformation formula due to Heine \cite[p. 325, XVIII]{heine} is a $q$-analogue 
of Euler's formula \eqref{tle}. 
We give an analogous proof using Theorem \ref{thm5}. 
  
\begin{thm}\label{thm6}
We have
\begin{equation}\label{teq}
\phi_{\a\b\g^{-1}}(x)
{}_2\phi_1\left({\a,\b\atop \g};x\right)
=
 {}_2\phi_1\left({\a^{-1}\g,\b^{-1}\g\atop \g};\a\b\g^{-1}x\right). 
\end{equation}
\end{thm}

Before the proof, we introduce another notation.  
\begin{dfn}
For $\a=q^a$, define an operator $\a^\d=q^{a\d}$ by 
$$\a^\d f(x)=f(\a x).$$ 
\end{dfn}
Then, one easily shows the following. 
\begin{lem}\label{l7.3}
There are identities of operators
\begin{align*}
(\a\b)^\d&=\a^\d\b^\d,\\
\D \a^\d &=\a \a^\d \D, \\
\a^\d g(x) \a^{-\d}&=g(\a x). 
\end{align*}
\end{lem}

\begin{proof}[Proof of Theorem \ref{thm6}]
Put $s=a+b-c$ and $\s=q^s=\a\b\g^{-1}$. 
By Theorem \ref{thm5}, ${}_2\phi_1\left({\a^{-1}\g,\b^{-1}\g\atop \g};x\right)$ 
and ${}_2\phi_1\left({\a,\b\atop \g};x\right)$ are solutions of the difference operators, respectively, 
\begin{align*}
\sD_1=&\D x^c\phi_{q\s^{-1}}\D+(1-q)[c-a][c-b]x^c\phi_{\s^{-1}}(qx)\D
\\&-[c-a][c-b]x^{c-1}\phi_{\s^{-1}}(qx).
\\ \sD_2 =& \D x^c\phi_{q\s}\D+(1-q)[a][b]x^c\phi_{\s}(qx)\D-[a][b]x^{c-1}\phi_{\s}(qx). 
\end{align*}
The right-hand side of \eqref{teq} is a solution of $\sD_1 \s^{-\d}$. 
We show the identity of operators 
\begin{equation}\label{e11}
\s^{2-c} \phi_\s(qx)\s^\d \sD_1 \s^{-\d} \phi_\s(x) = \sD_2.
\end{equation}
Then it follows that the left-hand side of \eqref{teq} is also a solution of $\sD_1 \s^{-\d}$. Comparing the initial values, which is easy, we obtain the theorem. 

We compute the left-hand side of \eqref{e11}. 
For the first term, using Lemmas \ref{l7.1}, \ref{l7.2}, \ref{l7.3} and \eqref{e10}, we have
\begin{align*}
& \s^{2-c} \phi_\s(qx)\s^\d \D x^c\phi_{q\s^{-1}}(x) \D \s^{-\d} \phi_\s(x) 
\\=&\s^{-c} \phi_\s(qx)\D \s^\d x^c\phi_{q\s^{-1}}(x)\s^{-\d} \D \phi_\s(x)  
\\=&\phi_\s(qx)\D x^c\phi_{q\s^{-1}}(\s x) \D \phi_\s(x) 
\\=&(\D \phi_\s(x) -\phi_\s'(x))x^c\phi_{q\s^{-1}}(\s x)(\phi_\s(qx)\D+\phi_\s'(x))
\\=&\D x^c \phi_\s(x)\phi_{q\s^{-1}}(\s x)\phi_\s(qx) \D
+\D x^c\phi_\s(x)\phi_{q\s^{-1}}(\s x)\phi_\s'(x)
\\& - x^c \phi_\s'(x)\phi_{q\s^{-1}}(\s x)\phi_\s(qx)\D -x^c \phi_\s'(x)^2\phi_{q\s^{-1}}(\s x)
\\=&\D x^c \phi_{q\s}(x)\D-[s]\D x^c \phi_\s(x)
+[s]x^c \phi_{\s}(qx)\D -[s]^2x^c\phi_{q^{-1}\s}(qx)
\\=&\D x^c \phi_{q\s}(x)\D-[s]\{(qx)^c\phi_\s(q x)\D + (x^c\phi_\s(x))'\}
\\& +[s]x^c\phi_{\s}(qx)\D -[s]^2x^c\phi_{q^{-1}\s}(qx)
\\=& \D x^c\phi_{q\s}(x)\D + [s](1-\g)x^c\phi_\s(qx)\D-[s](x^c\phi_\s(x))'-[s]^2x^c\phi_{q^{-1}\s}(qx) 
\\=& \D x^c\phi_{q\s}(x)\D + [s](1-\g)x^c\phi_\s(qx)\D
\\& -[s]^2(1-\g)x^c\phi_{q^{-1}\s}(qx)-[s][c]x^{c-1}\phi_{\s}(x). 
\end{align*}
For the second term, we have similarly
\begin{align*}
& \s^{2-c} \phi_\s(qx)\s^\d x^c \phi_{\s^{-1}}(qx)\D  \s^{-\d} \phi_\s(x) 
\\=&\s^{1-c} \phi_\s(qx)\s^\d x^c \phi_{\s^{-1}}(qx)  \s^{-\d}\D \phi_\s(x) 
\\=&\s x^c\phi_\s(qx)\phi_{\s^{-1}}(q\s x) (\phi_\s(qx) \D+\phi_\s'(x))
\\=&\s x^c \phi_\s(qx)\D-[s]\s x^c\phi_{q^{-1}\s}(qx)). 
\end{align*}
Finally for the last term, we have
\begin{align*}
& \s^{2-c} \phi_\s(qx)\s^\d x^{c-1}\phi_{\s^{-1}}(qx)  \s^{-\d} \phi_\s(x) 
\\ =& \s x^{c-1} \phi_\s(qx)\phi_{\s^{-1}}(q\s x)\phi_\s(x)
 =\s x^{c-1}\phi_\s(x). 
\end{align*}
Then, using 
\begin{align*}
&[s](1-\g)+(1-q)[c-a][c-b]\s=(1-q)[a][b], \\
&\phi_\s(qx)=(1-\s x)\phi_{q^{-1}\s}(qx)=\frac{1-\s x}{1-x}\phi_\s(x), 
\end{align*}
we obtain \eqref{e11}, hence the theorem. 
\end{proof}

\section*{Acknowledgement}
The author would like to thank Ryojun Ito, Hiroyuki Ochiai, Nobuki Takayama and Raimundas Vid\=unas for helpful discussions. 
This work is supported by JSPS Grant-in-Aid for Scientific Research: 18K03234. 

\ \\


\end{document}